\documentclass[10pt]{amsart}
\usepackage{amssymb,latexsym,amsmath}

\newcommand{\p}{{\partial}}

\newcommand{\Om}{{\Omega}}
\newcommand{\om}{{\omega}}

\newcommand{\Ll}{{\mathcal L}}

\newcommand{\Mm}{{\mathcal M}}

\newtheorem{theorem}{Theorem}[section]
\newtheorem{thm}[theorem]{Theorem}

\newtheorem{cor}[theorem]{Corollary}
\newtheorem{lemma}[theorem]{Lemma}

\newtheorem{prop}[theorem]{Proposition}

\newtheorem{defn}[theorem]{Definition}

\newtheorem{remark}[theorem]{Remark}

\numberwithin{figure}{section}
\numberwithin{equation}{section}
\numberwithin{table}{section}

\renewcommand{\Tilde}{\widetilde}

\newcommand{\TT}{{\mathbb T}}

\newcommand{\R}{{\mathbb R}}
\newcommand{\C}{{\mathbb C}}

\newcommand{\Z}{{\mathbb Z}}

\date{\today}
\address{Mathematics Department
1326 Stevenson Center
Vanderbilt University
Nashville, TN 37240}
\email{michael.j.chance@vanderbilt.edu}


\begin{document}
\bibliographystyle{plain}

\title{Degenerate Maxima in Hamiltonian Systems}
\author{Mike Chance}

\maketitle

\begin{abstract}
\noindent
In this paper we explore loops of non-autonomous Hamiltonian diffeomorphisms with degenerate fixed maxima.  We show that such loops can not have totally degenerate fixed global maxima.  This has applications for the Hofer geometry of the group of Hamiltonians for certain symplectic 4 manifolds and also gives criteria for certain 4 manifolds to be uniruled.
\end{abstract}

\begin{section}{Introduction and Main Results}
In this paper we examine Hamiltonian flows whose associated Hamiltonians have fixed maxima.  By this we mean points where a time dependent Hamiltonian attains a maximum for all time.  In  \cite{mc} McDuff proves several results for loops of Hamiltonian diffeomorphisms for which the fixed global maxima are nondegenerate.  The aim of this paper is to extend these results to the degenerate case. 

Such Hamiltonian flows may be viewed as generalizations of the autonomous case, while loops of this form are natural generalizations of Hamiltonian $S^1$ actions.  The existence of such actions gives useful information about the geometry of the underlying manifold, see e.g. \cite{eg, mt}.  Furthermore, both types of flows may be exploited to study the Hofer geometry of the Hamiltonian group.

Throughout the paper $(M, \omega)$ will be a closed, connected symplectic manifold.  $H_t : M \rightarrow \R$ will be a smooth family of Hamiltonians parameterized by $t \in [0,1]$ and $X_t^H$ will denote the associated Hamiltonian vector field defined by $\omega(X_t^H, \cdot) = dH_t$ and $\phi_t^H$ will the corresponding flow.

\begin{defn}\label{degdef}
A point, $x_0 \in M$, satisfying $\phi^H_t(x_0) = x_0, \forall t$ is called nondegenerate at $t_0$ if $\frac{d}{dt}|_{t=t_0}D\phi_t^H(x_0)v \ne 0$ for all $0 \ne v \in T_{x_0}M$, nondegenerate if it is nondegenerate for all time.  It is called totally degenerate if $D\phi_t^H(x_0) = Id$ for all values of $t$.
\end{defn}

\begin{defn}\label{fixedmaximum}
A point $x_0 \in M$ is called a fixed local maximum on U of $H_t$ if there exists a neighborhood $U \subset M$ of $x_0$ such that $H_t(x) \ge H_t(y)$ for all values of $t$ and $\forall y \in U$.  Similarly, $x_0 \in M$ is called a fixed global maximum if $H_t(x_0) \ge H_t(y)$ for all values of $t$ and $\forall y \in M$.  We will denote the collection of fixed local maxima $D_{max}$ and fixed global maxima $F_{max}$.
\end{defn} 

Our first result is the following theorem.

\begin{thm}\label{totallydegnonmax}
Given a loop of Hamiltonian diffeomorphisms, the collection of totally degenerate fixed local maximum points $D_{max} \subset M$ is open.
\end{thm}

In the event that our maximum is global, we prove the following consequence.

\begin{cor}\label{totallydegenerateintro}
Let $\{\phi_t^H\}$, $t \in [0, 1]$, be a nonconstant loop in $Ham(M, \omega)$ based at $Id$.  If $x_0$ is a fixed global maximum, then we must have $D\phi_t^H(x_0) \ne Id$ for some value of $t$.
\end{cor} 

Of course a similar statement holds for minima by simply considering the function $-H_t$.  Our proofs use methods of holomorphic curves and the requirement that the maximum is global in Corollary \ref{totallydegenerateintro} cannot be dropped.  This result allows us to then construct loops of Hamiltonian diffeomorphisms  with fixed nondegenerate global maxima.  Combining these constructions with results of Slimowitz \cite{sl} we obtain:

\begin{thm}\label{fourdimdeg}
Let $M$ be a symplectic manifold with $\dim M \le 4$.  If $\{\phi_t^H\}$ is any nonconstant loop of Hamiltonian diffeomorphisms with $x_0 \in M$ a fixed global maximum, then there is a nonconstant loop $\phi_t^K$ with $x_0$ still a fixed global maximum, which is an effective $S^1$ action near $x_0$.
\end{thm}

The dimensional restriction here is due to the fact that the homotopy results of Slimowitz have only been proved for $\dim \le 4$, although in principle those results should hold in all dimensions.

A symplectic manifold is called uniruled if some point class nonzero Gromov-Witten invariant does not vanish.  More specifically this means there exist $a_2, \dots, a_k \in H_*(M)$ so that
\begin{eqnarray*}
\langle pt, a_2, \dots, a_k\rangle^M_{k, \beta} \ne 0 \mbox{ for some } 0 \ne \beta \in H^S_2(M),
\end{eqnarray*}

\noindent
where $pt$ is the point class in $H_0(M)$.  We refer the reader to \cite{ms1} for details on the Gromov-Witten invariants.  In \cite{mc} McDuff uses the Seidel element and methods of relative Gromov-Witten invariants to show that manifolds admitting a loop of Hamiltonian diffeomorphisms with a fixed nondegenerate global maximum must be uniruled.  Thus, combining the results of McDuff with Theorem \ref{fourdimdeg} we have:

\begin{thm}\label{nondeguniintro}
If $\dim M \le 4$ and there exists a nonconstant loop of Hamiltonian diffeomorphisms with a fixed global maximum,  then $(M, \omega)$ is uniruled.
\end{thm}

McDuff relies heavily on the algebraic structure of the quantum homology of $M$ as well as the invertibility of the Seidel element.  While methods used in this paper are largely inspired by these, we rely solely on the geometric structures of a certain Hamiltonian bundle over $S^2$, as opposed to the algebraic information the bundle gives rise to.  

Given a path $\phi_t^H$, $t \in [0,1]$, the Hofer length is defined as
\begin{eqnarray*}
\Ll(\phi^H_t) = \int_0^1 \Big(\max_x H_t(x) - \min_x H_t(x)\Big)dt.
\end{eqnarray*}

This allows one to construct a nondegenerate Finsler metric on the group of Hamiltonian diffeomorphisms, $Ham(M, \omega)$, whose geometry has been studied extensively, see e.g. \cite{bp, kl, lm3, po}.  In particular, in \cite{lm3} Lalonde and McDuff show that if $\phi^H_t$ is a Hofer length minimizing geodesic then its generating Hamiltonian has at least one fixed global minimum and one fixed global maximum.  Thus Theorem \ref{nondeguniintro} implies the following:

\begin{thm}\label{hofintro}
Let $(M, \omega)$ be a closed, connected symplectic $4$-manifold, and suppose that $\gamma \in \pi_1(Ham(M, \omega))$ is nontrivial.  If there exists a representative $\{\phi^H_t\}$ of $\gamma$ which is Hofer length minimizing, then $(M, \omega)$ is uniruled.
\end{thm}

Of course this says nothing if the Hamiltonian group is simply connected.  In \cite{mc2} McDuff demonstrates that $\pi_1(Ham(M, \omega)) \ne 0$ if $M$ is a suitable two point blow up of any symplectic 4-manifold.  Thus, if $M$ is not uniruled (e.g. $\TT^4$, a $K3$ surface, or a surface of general type), this two point blow up is a 4-manifold for which there are nontrivial elements of $\pi_1(Ham(M, \omega))$ having no Hofer length minimizing representatives.

  This paper is organized as follows.  Section \ref{pos_semipos2} contains a discussion of positive and semipositive paths.  It also contains the proofs needed for Theorem \ref{fourdimdeg} assuming Corollary \ref{totallydegenerateintro}.  Section \ref{hambundles} contains a discussion of the Hamiltonian fibrations used.  As the machinery needed to prove Theorem \ref{totallydegnonmax} is discussed here, its proof is left to the end of this section.
 
 \begin{subsection}{Acknowledgements}
Much of this work is part of the author's thesis under the advisement of Dusa McDuff.  The author is deeply grateful to her for many valuable discussions and advice.  The author would also like to thank Ba\c{s}ak G\"{u}rel and Aleksey Zinger  for their helpful conversations and suggestions.
 \end{subsection} 
  
 \end{section}


\begin{section}{Positive and Semipositive Paths}\label{pos_semipos2}

\begin{subsection}{Positive and Semipositive Paths}
Consider $\R^{2n}$ with the standard symplectic structure $\omega$ and almost complex structure $J$.  Recall that $Sp(2n)$ consists of all matrices satisfying $A^TJA = J$, and its Lie algebra, \textbf{sp}$(2n)$, consists of matrices which satisfy $JAJ = A^T$.  Throughout, when in $\R^{2n}$, we use these structures and the metric given by, $g(\cdot, \cdot) = \omega(\cdot, J\cdot)$.

A differentiable path in $A_t \in Sp(2n)$ is called \emph{positive} if it satisfies
\begin{eqnarray*}
\frac{d}{dt}A_t = JQ_tA_t
\end{eqnarray*}
\noindent 
where $Q_t$ is a positive definite symmetric matrix for each $t$.  Such paths are natural generalizations of circle actions near maxima of the corresponding autonomous Hamiltonian.  The linearization of a Hamiltonian has the form
\begin{eqnarray*}
H_t(x) = const - \frac{1}{2}\langle x, Q_tx\rangle,
 \end{eqnarray*}

\noindent
and if it is semipositive, it corresponds precisely with the linearized flow near a maximum of some Hamiltonian.  The simplest example of such a path is the counter clockwise rotation $A_t = e^{2\pi kJt}$, with $k > 0$.  Here $Q_t = 2\pi kI$.  In the event that $Q_t$ is symmetric, but only positive semidefinite (i.e., $Q_t$ could have eigenvalues of zero for certain values of $t$), the path is called \emph{semipositive}.  In \cite{sl} Slimowitz proves the following:  

\begin{thm}\label{slimowitz1}(Slimowitz)
Let $n=1,2$ and let $A_t \in Sp(2n)$ be a positive loop.  Then $A_t$ can be homotoped through positive loops to a circle action.
\end{thm}

As mentioned, in principle this should be true in all dimensions, but the details have only been worked out for these cases.  Slimowitz shows further that

\begin{lemma}\label{homotop}(Slimowitz)
In $Sp(4)$, any two loops of matrices of the form 
\begin{eqnarray*}
\left( \begin{array}{cc} 
e^{2\pi b_iJt} & 0  \\ 
0 & e^{2\pi d_iJt}  \\  
\end{array} \right) 
\end{eqnarray*}

\noindent
for $i=1,2$ and $t \in [0,1]$ are homotopic through positive loops provided $b_i, d_i \ge 1$ and $b_1+d_1 = b_2+d_2$.
\end{lemma}
\end{subsection}


\begin{subsection}{Proof of Theorem \ref{fourdimdeg}}
In our setting we wish to consider Hamiltonians on manifolds.  Fixed maxima must be fixed points of the associated flow for all time (i.e., $\phi^H_t(x) = x, \forall t$).  Choosing a Darboux chart around such a point $x$, $H_t$ may be written as
\begin{eqnarray}\label{localham}
H_t(x) = const - \frac{1}{2}\langle x, Q_tx\rangle + O(\|x\|^3)
\end{eqnarray}

\noindent
and as before we will call the path (semi)positive if $Q_t$ is positive (semi)definite. As the point $x$ is only assumed to be a fixed maximum, it may be degenerate and we may only assume $Q_t \ge 0$ and the flow of its linearization is a semipositive path.  

In \cite{mc} McDuff proved the following result:

\begin{lemma}(McDuff)
Suppose the loop $\gamma$ in $Ham(M, \omega)$ has a nondegenerate fixed maximum at $x_0$.  Suppose also that the linearized flow at $x_0$ is homotopic through positive paths to a linear circle action.  Then $\gamma$ is homotopic through loops of Hamiltonian diffeomorphisms with fixed maximum at $x_0$ to a loop $\gamma'$ that is the given circle action near $x_0$.
\end{lemma}

\noindent
Thus given a degenerate global maximum, we must construct a new loop with a nondegenerate global maximum, and then apply the results of McDuff and Slimowitz to obtain Theorem \ref{fourdimdeg}.

The first results deal with the case when $Q_{t_0} > 0$ for some $t_0$, and thus is a positive path for some $\epsilon$ time.  We describe a method of ``spreading out the positivity" to homotop our path to a new one which is positive on all of $[0,1]$.  We do so in such a way that, if $x$ is a maximum of $H_t$ on some set $V$, it will remain a maximum of the new Hamiltonian on $V$.

\begin{lemma}\label{firstspread}
Let $\{\phi_t^H\} \subset Ham(M, \omega)$ for $t \in [0,1]$ be a path of Hamiltonian diffeomorphisms whose generating function, $H_t$, has a fixed local maximum at $x_0$.  Let $-\frac{1}{2}\langle x, Q_t x \rangle$ be the quadratic part of $H_t$, 
and let $I^+ = \{t \in [0,1] | Q_t > 0\}$.  If $\emptyset \ne I^+ \ne [0,1]$ choose $t_0 \in I^+$ and $t_1 \notin I^+$. Then the path may be homotoped through semipositive paths with fixed endpoints to a new one, whose quadratic part is positive in a $\delta'>0$ neighborhood of $t_1$ and remains so in $I^+$.  Furthermore, if $x_0$ was a maximum of $H_t$ on a neighborhood $V$ of $x_0$ for all $t$, it will remain a maximum of $F_t$ on $V$ for all $t$, and $\delta'$ will depend only on the initial neighborhood of $t_0 \in I^+$.
\end{lemma}

\begin{proof}
For the purposes of this proof, consider $t$ as a variable in $\R/\Z$.  Let $\delta$ be such that $Q_t>0$ for $|t-t_0|<\delta$.  Note that we must have $|t_1-t_0|\ge \delta$.  We will show that we may take $\delta' = \delta/3$.

Let $a$ be smaller than any of the eigenvalues of $Q_t$ for $|t-t_0|<\delta/2$ and let $b >> 1$.  Define a function $\alpha : \R_{\ge 0} \rightarrow \R$ satisfying: $\alpha'(r) \ge 0$, $\alpha(r) = br-a$ for $r < a/2b$, and $\alpha(r) = 0$ for $r > a/b$.  Next consider the autonomous Hamiltonian $K$ defined on $\R^{2n}$ given by $K(x) = \alpha(\| x \|)\| x \|^2$.  Also, let $\beta : [0,1] \rightarrow [0,1]$ be a smooth nonincreasing function which is $1$ on $[0, \delta/3]$, and $0$ on $[\delta/2, 1]$.

For each value of $0 \le s \le 1$ define a function:
\begin{eqnarray*}
K_{s,t} = \left\{ \begin{array}{ll}
0 & \textrm{for $t < t_1 - \delta/2$}\\
s\beta(|t - t_1|) K & \textrm{for $t_1 - \delta/2 \le t \le t_1 + \delta/2$}\\
0  & \textrm{for $t_1 + \delta/2 \le t \le t_0 - \delta/2$}\\
-s\beta(|t - t_0|)K & \textrm{for $t_0 - \delta/2 \le t \le t_0 + \delta/2$}\\
0 & \textrm{for $t _0 + \delta/2 \le t$}
\end{array} \right.
\end{eqnarray*}

\noindent 
For each value of $s$, this time-dependent function will generate a smooth path of Hamiltonian diffeomorphisms, $\{\psi^K_{s, t}\}$.   Since any perturbation from the identity map is eventually undone, the path will satisfy $\psi^K_{s, 0} = \psi^K_{s,1} = Id$, regardless of the values of $t_0$ and $t_1$, and thus will always be a loop.  

We now consider the composition $\phi^F_{s, t} = \psi^K_{s, t} \circ \phi^H_t$, and note that the corresponding time dependent family of functions $F_{s,t}$ are given by the formula 
\begin{eqnarray}\label{formula}
F_{s,t} = K_{s,t} \# H_t = K_{s,t} + H_t \circ (\phi^K_{s, t})^{-1}.
\end{eqnarray}

We now claim that for a suitable choice of $b$, our path $\phi^F_{s, t}$ is positive for $t \in I^+$ and $|t-t_1|<\delta/3$.  To show positivity, we need only show that Hamiltonian has non-degenerate quadratic part at $x_0$.  Fix $s$ and $t$ with $|t-t_1|<\delta/3$ and $v \in \R^{2n}$, and consider the limit
\begin{eqnarray*}
\lim_{r \rightarrow 0}\frac{\Big(K_{s,t}(rv) + H_t \circ (\psi^K_{s,t})^{-1}(rv)\Big)}{\|rv\|^2} & = & -a\beta(|t-t_1|)s + \lim_{r \rightarrow 0} \frac{H_t \circ (\phi^K_{s, t})^{-1}(rv)}{\|rv\|^2} \\
& \le & -a\beta(|t-t_1|)s \\
& < & 0
\end{eqnarray*}

\noindent 
where the inequality and subsequent minus sign on the right are explained by our convention of the quadratic portion actually being negative semidefinite.

Calling $Q'_t$ the quadratic portion of $F_t$, $Q'_t > 0$ for $|t-t_0|<\delta/2$, since $a$ was chosen smaller than any of the eigenvalues of $Q_t$ here.  To see that $Q'_t>0$ on the rest of $I^+$, we note that $\beta =0$ in this region, and 

\begin{eqnarray*}
\lim_{r \rightarrow 0}\frac{\Big(K_{s,t}(rv) + H_t \circ (\psi^K_{s,t})^{-1}(rv)\Big)}{\|rv\|^2} & = &  \lim_{r \rightarrow 0} \frac{H_t \circ (\phi^K_{s, t})^{-1}(rv)}{\|rv\|^2} \\
& < & 0
\end{eqnarray*}

\noindent
As $s$ and $t$ are fixed while taking this limit, the inequality holds by the definition of $I^+$.

Finally, our perturbed function will remain a maximum in $V$ for $|t - t_0| \ge \delta'$, and by choosing $b$ large enough (the choice depends on the third order terms of the initial $H_t$), it will remain a maximum for $|t - t_0| < \delta'$, as well.
\end{proof}

\begin{prop}\label{secondspread}
Let $\{\phi^H_t\}$ for $t \in [0, 1]$ be a path of Hamiltonian diffeomorphisms based at $Id$ with generating function $H_t$.  Suppose $x_0$ is a maximum of $H_t$ on some neighborhood $V$ of $x_0$, for all $t$.  Letting $Q_t$ be as in Lemma \ref{firstspread}, if $Q_{t_0} > 0$ for some $0 \le t_0 \le 1$, then $\{\phi^H_t\}$ may be homotoped through semipositive paths with fixed endpoints to a new path $\{\phi^F_t\}$ whose associated quadratic portion is strictly positive for all $t \in [0,1]$.  Furthermore $x_0$ will be a maximum of $F_t$ on $V$ for all $t$, as well.
\end{prop}

\begin{proof}
As the $\delta' > 0$ value from Lemma \ref{firstspread} depended only on the neighborhood of $t_0$ in $I^+$, we may carry out the process a finite number of times to homotop our path through semipositive paths with fixed endpoints to one which is positive for all $t$.  Furthermore, by construction, $x_0$ remains a maximum on $V$ throughout.
\end{proof}

The next result deals with the case when $Q_t \ngtr 0$ for any $t$, but $Q_{t_0} \ne 0$ for some $t_0$.  Thus while it is never a positive path, it is positive in at least one direction at time $t_0$.

\begin{prop}\label{thirdspread}
Let $\{\phi_t^H\} \subset Ham(M, \omega)$ for $t \in [0,1]$ be a path of Hamiltonian diffeomorphisms with generating function $H_t$.  Let $x_0$ be a maximum of $H_t$ on some neighborhood $V$ of $x_0$.  If $D\phi^H_t \ne Id$ for some $t_0$, then there is a new path, $\{\phi^K_t\}$, whose associated Hamiltonian, $K_t$, is nondegenerate at $x_0$ and for $t = t_0$.  Furthermore, $\phi_t^K$ can be chosen to be homotopic to $\{(\phi_t^H)^m\}$ for some $m \le 1 + dim(ker(D\phi^H_{t_0} - Id))$.  Furthermore, $x_0$ will remain a maximum of $K_t$ on $V$.
\end{prop}

\begin{proof}
Throughout, for convenience of notation, we explicitly work in $\R^{2n}$ and the linearization of $\phi^H_t$.  We refer to the linearization as the path $A_t \in Sp(2n, \R)$, and note that it satisfies $A_0 = A_1 = Id$ and $\frac{d}{dt}A_t(x) = JQ_t(x)A_t(x)$ with $Q_t \ge 0$ and symmetric.  Let $t_0$ be such that $Q_{t_0} \ne 0$.

Identify $\R^{2n} = E_0 \oplus E_1$ with $E_0 = ker(Q_{t_0})$ and $E_1$ the sum of eigenspaces of $Q_{t_0}$ with nonzero eigenvalues.  We first consider the case when $J(E_0) = E_0$.  Choose $v \in E_1$ to be an eigenvector for $Q_{t_0}$, and let $0 \ne w \in E_0$.  Split $\R^{2n} = \R^4 \oplus \R^{2n-4}$ with $\R^4$ spanned by $\{v, w, Jv, Jw\}$ and $\R^{2n-4} = (\R^4)^{\omega}$ its symplectic orthogonal. Define $B \in Sp(2n)$ by
\begin{eqnarray*}
BA_{t_0}^{-1}w = v, & & BA_{t_0}^{-1}Jw = Jv\\
BA_{t_0}^{-1}v  =  -w, & & BA_{t_0}^{-1}Jv = -Jw, \\
BA_{t_0}^{-1}|_{\R^{2n-4}} = Id. & &
\end{eqnarray*}

\noindent
Let $B_s \in Sp(\R^{2n})$ for $s \in [0,1]$ satisfy $B_0 = Id$ and $B_1 = B$, and let $f_s : \R^{2n} \rightarrow \R^{2n}$ be a family of symplectomorphisms fixing the origin and supported in a small neighborhood of it satisfying $Df_s(0) = B_s$, where $Df_s(0) : \R^{2n} \rightarrow \R^{2n}$ is the derivative at the origin.  We may assume that $x_0$ is a maximum of $H_t$ throughout the support of the family $f_s$.  We now wish to consider the family of paths given by:
\begin{eqnarray*}
\phi^K_{s, t} = \phi^H_t (f_s^{-1} \phi^H_t f_s)
\end{eqnarray*}

\noindent
which will provide a homotopy from $\big(\phi_t^H\big)^2$ to $\phi_{1,t}^K$.  For each $s$ this will remain a Hamiltonian flow, and will be generated by:
\begin{eqnarray*}
K_{s,t} = H_t + H_t(f_s \circ (\phi_t^H)^{-1}).
\end{eqnarray*}

By our choice of $f$, $x_0$ will remain a constant maximum of $\phi_{s,t}^K$ for all values of $s$.  To  determine the degeneracy of our maximum, we simply differentiate:
\begin{eqnarray*}
\frac{d}{dt} D\phi_{s,t}^H(0) & = & \frac{d}{dt} A_t B_s^{-1} A_t B_s \\
                                                 & = & \dot{A}_tB_s^{-1} A_t B_s + A_t B_s^{-1} \dot{A}_t B_s \\
                                                 & = & J Q_t (A_t B_s^{-1}A_tB_s) + A_tB_s^{-1}JQ_tA_t B_s \\
                                                 & = & J\Big( Q_t + (B_sA_t^{-1})^TQ_t (B_sA^{-1}_t) \Big) (A_tB_s^{-1}A_tB_s).
 \end{eqnarray*} 

Call $\Gamma_{s,t} =  (B_1A_{t}^{-1})^TQ_{t} (B_1A^{-1}_{t})$.  Since both $Q_t$ and $\Gamma_{s,t}$ remain symmetric and nonnegative for all values of $s$ and $t$, so is their sum.  Furthermore, if $Q_t + \Gamma_{s,t}$ has any kernel, it must be contained in the kernel of $Q_t$.  Thus we need only check that it is nondegenerate in the $v, w$ plane when $s=1$ and $t=t_0$.  Abusing notation, call $\Gamma_{1,t_0} = \Gamma$.  We compute:

\begin{eqnarray}\label{expand}
\langle v+aw, (Q_{t_0} + \Gamma)(v + aw)\rangle & = & \langle v, Q_{t_0}v\rangle + \langle v, \Gamma v\rangle +a\langle v, \Gamma w\rangle \\
\notag                                                                                               &    & +a\langle w, \Gamma v\rangle + a^2\langle w, \Gamma w\rangle \\
\notag & = & (1+a^2)\langle v, Q_{t_0}v\rangle \\
\notag & > & 0.
\end{eqnarray}

To see that we created no new kernel, let $ u \in \R^{2n}$.  Then
\begin{eqnarray*}
\langle u, (Q_{t_0} + \Gamma)u\rangle = \langle u, Q_{t_0}u\rangle + \langle u, \Gamma u\rangle
\end{eqnarray*}

\noindent
with both matrices being nonnegative.  Thus the sum can only be zero if $\langle u, Q_{t_0}u\rangle = 0$.

In the case when $J$ does not preserve $E_0$, we may choose $w \in E_0$ so that $\langle Jw, Q_{t_0}Jw\rangle > 0$.  In this case, setting $v = Jw$ we define
\begin{eqnarray*}
BA_{t_0}^{-1}w = v, & & BA_{t_0}^{-1}v  =  -w,\\
BA_{t_0}^{-1}|_{\R^{2n-2}} = Id. & &
\end{eqnarray*}

As \eqref{expand} remains the same, the remainder of the proof is identical to the previous case. 
\end{proof}

\begin{cor}
Let $x_0 \in M$ be a fixed local maximum on $U$ of a family of Hamiltonians $H_t$, such that $D\phi_t^H(x) \ne Id$ for some $t_0$.  Then we may construct a new path $\phi^F_t$ which is positive for all $t \in [0,1]$ and such that if $H_t(x_0) \ge H_t(y)$ for some $y \in U$, then $F_t(x_0) \ge F_t(y)$.  Furthermore, $\phi_t^F$ is an iterate of $\phi_t^H$.
\end{cor}

Assuming Corollary \ref{totallydegenerateintro}, combining the above results finishes the proof of Theorem \ref{fourdimdeg}.

\begin{remark}
One may note that there is a slight error in \cite{mc}.  Proposition 1.4 of that paper states that if $\dim M \le 4$, and the loop $\{\phi_t^H\}$ has a nondegenerate fixed global maximum, then it can be homotoped so that it is an effective $S^1$ action near the maximum.  This is actually only true if $\dim M = 4$.  If $\dim M = 2$, the existence of such a loop forces the manifold to be $S^2$, and thus there certainly exists an effective $S^1$ action.  However a two time rotation is not homotopic to an effective action. 
\end{remark}

\end{subsection}
\end{section}


\begin{section}{Hamiltonian Fibrations}\label{hambundles}

This section contains proofs of Theorem \ref{totallydegnonmax} and Corollary \ref{totallydegenerateintro}.  We begin with a discussion of the Hamiltonian fibration that will be used.  One may note that the autonomous case is much more straightforward, even if not restricted to global circle actions.  It is a standard result that given an autonomous path of Hamiltonian diffeomorphisms, around any totally degenerate fixed point, there is an entire neighborhood containing no nontrivial periodic orbits (see e.g. Lemma 12.27, \cite{ms1}), and thus cannot be a loop.

\begin{subsection}{Hamiltonian Bundles Over $S^2$}
Given $(M, \omega)$ and any loop of Hamiltonian diffeomorphisms, $\{\phi_t^H\}$, there is an associated Hamiltonian fibration $P \rightarrow S^2$ with fiber symplectomorphic to $(M, \omega)$.  Throughout we use the standard almost complex structure on $S^2$, which we call $j$.  Begin with two copies of $M \times D_{\pm}$, where $D_{\pm}$ denotes two different copies of the unit disk with opposite orientations.  Define the equivalence relation:
\begin{eqnarray}\label{bundle}
P: = M\times D_+\cup M\times D_-/\!\sim, \;\;(\phi^H_t(x),e^{2\pi it})_+=
(x,e^{2\pi it})_-. 
\end{eqnarray}

Since $\phi_0^H(x) = \phi_1^H(x) = x, \forall x \in M$, the two copies of $D$ glue together along their boundaries to give a copy of $S^2$, and $P \rightarrow S^2$ will be a fibration with fiber diffeomorphic to $M$.  Denote the projection map by $\pi : P \rightarrow S^2$.  The vertical tangent bundle here is given by $T^{Vert}P = ker(D\pi) \subset TP$.  Because the fibers are symplectic, they have Chern classes and we denote the vertical first Chern class by $c_1^{Vert}$. 

Define a symplectic form $\Om$ on $P$ by
\begin{eqnarray}\label{form}
& \Om_-: = \om + \delta d(r^2)\wedge dt, \mbox{ on } M \times D_- & \\
\notag & \Om_+: = \om + \Big(\kappa(r^2,t)d(r^2) - d(\rho(r^2)H_t)\Big)\wedge dt \mbox{ on } M \times D_+ &
\end{eqnarray}

\noindent
where we have used normalized polar coordinates $(r, t)$ on $D$ with $t: = \theta / 2\pi$.  Here $\rho(r^2)$ is a nondecreasing function that equals 0 near 0 and 1 near 1, and $\delta > 0$ is a small constant.  As long as $\kappa(r^2, t) = \delta$ near $r=1$, these two forms will fit together to give a closed form on $P$.  To be symplectic, $\Om$ must be nondegenerate, but this can be seen to happen iff $\kappa(r^2, t) - \rho '(r^2)H_t(x) > 0, \forall (r,t) \in D_{\pm}$ and $x \in M$.  

$\Om$ restricted to $T^{Vert}P$ is nondegenerate.  Thus $\Om$ gives a connection $2$-form on $P$, and we have a well defined horizontal distribution, which we will denote by $T^{Hor}P$.  To be more precise:
\begin{eqnarray}\label{hor}
T^{Hor}_pP = \{v \in T_pP | \Om(v, w) = 0, \forall w \in kerD\pi(p) \} 
\end{eqnarray}

\begin{defn}\label{compatible} (\cite{ms2} $\S 8.2$) An almost complex structure, $\Tilde{J} : TP \rightarrow TP$ will be called compatible with the fibration if the following conditions are met:
\begin{enumerate}
\item $\pi : P \rightarrow S^2$ is holomorphic
\item $\Tilde{J} |_{T^{Ver}_zP}$ is tamed by $\om, \forall z \in S^2$
\item $\Tilde{J}(T^{Hor}P) \subset T^{Hor}P$.
\end{enumerate}
\end{defn}

\noindent
Note here that varying $\kappa$ in \eqref{form} does not affect the horizontal distribution defined in \eqref{hor}.

By our choice of $\Omega$, $T^{Hor}(M\times D_-)$ is spanned by $\p_r$ and $\p_t$, and  $T^{Hor}(M\times D_+)$ is spanned by the vectors $\p_r$ and $\p_t - X_t^H$ at each point, and conditions (1) and (3) completely determine $\Tilde{J}$ on $T^{Hor}P$.  Also, because $\Tilde{J}$ is tamed by $\Omega$, the bilinear form 
\begin{eqnarray*}
g_{\Tilde{J}}(v, w) = \frac{1}{2}\big(\Omega( v, \Tilde{J} w) + \Omega (w, \Tilde{J}v)\big)
\end{eqnarray*}

\noindent
defines a Riemannian metric on $P$, with associated Levi-Civita connection, $\nabla$.  To obtain a connection which will preserve $\Tilde{J}$ we use (\cite{ms2} $\S$ 3.1)
\begin{eqnarray} \label{connection}
\Tilde{\nabla}_vX = \nabla_vX - \frac{1}{2}\Tilde{J}(\nabla_v\Tilde{J})X.
\end{eqnarray}

This bundle contains lots of sections.  To see this explicitly, choose any $x \in M$.  As $\{\phi^H_t\}$ is a loop, every point gives rise to a contractible 1-periodic orbit.   A contraction of the orbit $\{\phi^H_t(x)\}$ is a map from the unit disk $f: D \rightarrow M$ with $f(e^{2\pi i t}) = \phi_t^H(x)$.  Thus an explicit formula for a section $s:S^2 \rightarrow P$ would be
\begin{eqnarray}\label{section}
& s(r,t) = x \times (r,t) \mbox{, on } D_- & \\
\notag & s(r,t) = f(r,t) \times (r,t) \mbox{, on } D_+.
\end{eqnarray}

In the event that $x_0$ is fixed by $\phi_t^H$ for all time, we may choose $f(r,t)$ to be constant.  We will denote such a constant section by $s_0$.  These sections have particularly nice properties as they are holomorphic with respect to compatible almost complex structures on $P$.  Another important property holds when $x_0$ is a fixed maximum of $H_t$.  We now argue as in McDuff (\cite{mc}, Proposition 2.11) as well as McDuff and Tolman (\cite{mt}, Lemma 3.1).

\begin{lemma}\label{smallest}
Suppose that $x_0 \in M$ is a fixed maximum on the open set $U \subset M$ of a loop of Hamiltonian diffeomorphisms for all time, and consider the constant section $s_0:(r,t) \mapsto x_0 \times (r,t)$.  Then, given any $\Tilde{J}$ compatible with the fibration, the only nearby holomorphic sections in class $[s_0]$ are constant ones, and are parameterized by elements of a component of the fixed local maximum set, $D_{max}$ for $H_t$.
\end{lemma}

\begin{proof}
We use the symplectic form given by \eqref{form}.  At a point in the image of our section $u:S^2 \rightarrow P$, split $TP = T^{Vert}P \oplus T^{Hor}P$, and write elements of $T_{u(r,t)}P$ as $v +h$.  If $u$ is sufficiently close to $s_0$, then it must only pass through our neighborhood $U$. We compute:

\begin{eqnarray*}
\Omega(v+h, \Tilde{J}(v+h)) & = & \omega(v,\Tilde{J}v) + \Omega(h,\Tilde{J}h) \ge  \Omega(h, \Tilde{J}h) \\
                                       & \ge & 2r\big(\kappa(r^2, t) - \rho'(r^2) \max_{x\in U}H_t(x)\big)dr \wedge dt(h, \Tilde{J}h).
\end{eqnarray*} 

The first inequality is an equality only if the curve is horizontal, and the second is an equality only if the section is contained in the same component of $D_{max} \times S^2$ as $x_0$.  Since another curve representing the same class as $[s_0]$ must have the same symplectic area, it must be constant.
\end{proof}

Note that a slight adjustment to the previous argument gives the following.

\begin{lemma}\label{smallest2}
Suppose that $x_0 \in M$ is a fixed global maximum on $M$ of a loop of Hamiltonian diffeomorphisms for all time, and consider the constant section $s_0:(r,t) \mapsto x_0 \times (r,t)$.  Then, given any $\Tilde{J}$ compatible with the fibration, the only holomorphic sections in class $[s_0]$ are constant ones, and are parameterized by elements of the fixed global maximum set, $F_{max}$ for $H_t$.
\end{lemma}

\end{subsection}


\begin{subsection}{Totally Degenerate Maxima}
Let $\phi_t^H$ be a loop in $Ham(M, \omega)$ based at $Id$, with $D_{max}$ the fixed local maxima set. We show that this set must be open.

The compatibility conditions of Definition \ref{compatible} do not determine $\Tilde{J}$ on $T^{Vert}P$, so we now construct one explicitly.  In our case, we wish the almost complex structure we construct to be regular for a constant section through $D_{max}$.

First, choose a Darboux chart around $x_0 \in D_{max}$, call it $U$, and identify it with a neighborhood of $0 \in \R^{2n}$.  Using the standard $\{x_i, y_i\}$ coordinates on $\R^{2n}$ and the standard $J$, choose an almost complex structure on $M$ which is the pullback of $J$ on $U$, and refer to it as $J_0$.  

Take $\Tilde{J_0}^{Vert}$ on $T^{Vert}(M \times D_-)$ to be $J_0$.  This forces $\Tilde{J}^{Vert} = (\phi^H_t)_*J_0$ on $M \times \p D_+$, and we must extend this to the rest of $T^{Vert}(M \times D_+)$.  In our coordinates on $U$,  $(\phi^H_t)_*J_0$ will be given by conjugation by $D\phi_t^H(x)$, so that at a point $x$, we have
\begin{eqnarray*}
(\phi^H_t)_*J_0 = (D\phi_t^H(x))^{-1} \circ J_0 \circ D\phi_t^H(x)
\end{eqnarray*}

\noindent 
where we have realized $D\phi_t^H(x)$ as a loop of maps $D\phi_t^H: U \rightarrow Sp(2n)$ based at the constant map $U \mapsto Id$.  Since $D\phi_t^H(0) = Id$ for all time, we may also assume our initial neighborhood $U$ is small enough that there is a loop of maps $Y_t : U \rightarrow \textbf{sp}(2n)$ based at the constant map $U \mapsto 0$ satisfying
\begin{eqnarray*}
\exp(Y_t(x)) = D\phi_t^H(x)
\end{eqnarray*}

\noindent
where $\exp$ is the standard exponential map from $\textbf {sp}(2n) \rightarrow Sp(2n)$.
Letting $\beta :[0,1] \rightarrow [0,1]$ be a smooth, nondecreasing function which is 0 near 0 and 1 near 1, we may consider the family of maps $Y_{r,t}: U \rightarrow \textbf{sp}(2n)$ given by $Y_{r,t}(x) = \beta(r)Y_t(x)$.  By our choice of $\beta$, $Y_{r,t}(x) = 0$ for $r$ close to $0$ and we may now consider this as a family of maps smoothly parameterized by $D_+$.  We now extend our almost complex structure to all of $U \times D_+$ by the formula:
\begin{eqnarray}\label{jvert}
\Tilde{J}^{Vert}(x \times (r,t)) = \exp(Y_{r,t}(x))^{-1} \circ J_0 \circ \exp(Y_{r,t}(x)),
\end{eqnarray}

\noindent
where $x \times (r,t) \in U \times D_+$.

Finally extend $\Tilde{J}^{Vert}$ to the rest of $T^{vert}(M \times D_+)$ in a way compatible with the fibration (see \cite{ms2}, \S 8.2), and take $\Tilde{J} = \Tilde{J}^{Vert} \oplus \Tilde{J}^{Hor}$.  We now claim that the $\Tilde{J}$ just constructed is a regular almost complex structure for our constant maximum section.

\begin{lemma}\label{vanishing}
Let $\xi \in \Omega^0(S^2, s_0^*(TP))$ be any vector field along $s_0$.  Then $\nabla_{\xi}\Tilde{J} = 0$.
\end{lemma}

\begin{proof}
Given a section $\xi$ of $TP$ defined in a neighborhood of $Im(s_0)$, we may write it as $v_{\xi} + h_{\xi}$ where $v_{\xi}$ is a section of $T^{Vert}P$ and $h_{\xi}$ a section of $T^{Hor}P$, both defined in a small neighborhood of $Im(s_0)$.  We consider $\Tilde{J}^{Hor}$ and $\Tilde{J}^{Vert}$ separately.  

It is clear that if $h$ is tangent to $Im(s_0)$, then $\nabla_h \Tilde{J}^{Hor} = 0$.  If $v \in T^{Vert}(M \times D_-)$, one also has $\nabla_v \Tilde{J}^{Hor} = 0$ along $x_0 \times D_-$.  If $x_0 \times (r,t) \in x_0 \times D_+$, then because $\nabla$ is Levi-Civita, we must have $\nabla_v(\p_t - X) = a_{r,t}\p_r$ and $\nabla_v \p_r = -a_{r,t}(\p_t - X)$ with $a_{r,t} \in \R$.  But then using the identity 
\begin{eqnarray*}
(\nabla_v \Tilde{J}^{Hor})(X) = \nabla_v(\Tilde{J}^{Hor}(X)) - \Tilde{J}^{Hor}(\nabla_v(X))
\end{eqnarray*}

\noindent
as well as the Leibniz rule, one can easily see that $\nabla_v\Tilde{J}^{Hor} = 0$ along $x_0 \times D_+$, as well.  Thus we have $\nabla_{\xi}\Tilde{J}^{Hor} = 0$ for any $\xi \in \Omega^0(S^2, u^*(TP))$.

Similar rationale holds to show $\nabla_h\Tilde{J}^{Vert} = 0$ for $h$ tangent to $Im(s_0)$, and $\nabla_v\Tilde{J}^{Vert} = 0$ for $v \in T^{Vert}P$ along $x_0 \times D_-$.  Thus we need only concern ourselves with the value of $\nabla_v\Tilde{J}^{Vert}$ at points in $U \times D_+$ with $U$ a neighborhood of $x_0$.   Locally, we may expand $\phi_t^H$ about $x_0$, so that
\begin{eqnarray}\label{localphi}
\phi^H_t(x) =  x + \sum_{i \le j} A_{i,j}(t)x_ix_j + O(\|x\|^3)
\end{eqnarray}

\noindent
with $A_{i,j}(t)$ a time-dependent loop of vectors in $\R^{2n}$ and the higher order terms also depending on time.  Since $x_0$ is a totally degenerate maximum, we may write $|H_t(x) - H_t(0)| \le C\|x\|^4$ in our neighborhood for some $C$, and thus $\|X_t^H(x)\| \le C'\|x\|^3$ in our neighborhood for some $C'$.  We now use the fact that 
\begin{eqnarray*}
X_t^H(\phi^H_{t_0}(x)) & = & \frac{d}{dt}\phi_t^H(x)|_{t=t_0} \\
                                     & = &  \sum_{i \le j} \Big(\frac{d}{dt}A_{i,j}(t)|_{t=t_0}\Big)x_ix_j + O(\|x\|^3)
\end{eqnarray*}

\noindent
for every $0 \le t_0 \le 1$.  In order for $\|\dfrac{d}{dt}\phi^H_t(x)\|  = \|X_t^H(\phi_t^H(x)) \| \le C'\|x\|^3$, we must have each $A_{i,j}(t)$ a constant function of $t$.  As $\phi_0^H = Id$, $A_{i,j}(t) = 0$ for all $t$.  Thus,
\begin{eqnarray*}
\phi^H_t(x) = x  + O(\|x\|^3) \\
D \phi_t^H(x) = Id + O(\|x\|^2).
\end{eqnarray*}

Since $D\phi_t^H(x) = \exp(Y_t(x))$ we have $\nabla_v\exp(Y_t(x)) = 0$, and it is easy to see that $\nabla_v\exp(Y_{r,t}(x)) = 0$ also.  Finally, since $\exp(Y_{r,t}(x)) = Id$ along our section, we may say 
\begin{eqnarray*}
\nabla_v\Big((\exp(Y_{r,t}(x))^{-1}\circ J_0 \circ \exp(Y_{r,t})\Big) = 0
\end{eqnarray*}

\noindent
along our section, as well.  Thus $\nabla_{\xi}\Tilde{J}^{Vert} = 0$, for any $\xi \in \Omega^0(S^2, u^*(TP))$.
\end{proof}

\begin{prop}\label{regular}
Let $\phi^H_t, t \in [0,1]$ be a loop of Hamiltonian diffeomorphisms based at $Id$.  Let $D_{max}$ be the set of fixed local maxima of $H_t$, and suppose that $D\phi_t^H(x_0) \equiv Id$ for some $x_0 \in D_{max}$ and for all values of $t$.  Let $s_0$ denote the constant section through $x_0$ and let $\Tilde{J}$ be as constructed above.  Then $s_0$ is a regular $\Tilde{J}$ holomorphic map.
\end{prop}

\begin{proof}
For $\Tilde{J}$ to be regular for $s_0$, the differential, 
\begin{eqnarray*}
D_{s_0} : \Omega^0(S^2, s_0^*(TP)) \rightarrow \Omega^{0,1}(S^2, s_0^*(TP))
\end{eqnarray*}

\noindent
which maps smooth sections of $s_0^*(TP)$ to $\Tilde{J}$ antiholomorphic $s_0^*(TP)$ valued 1-forms on $S^2$, must be surjective.  An explicit formula for $D_{s_0}$ evaluated at $\xi \in \Omega^0(S^2, s_0^*(TP))$ is given by:
\begin{eqnarray}\label{vertdiff}
D_{s_0}\xi = \frac{1}{2}\big(\Tilde{\nabla}\xi + \Tilde{J}(s_0)\Tilde{\nabla}\xi \circ j\big) + \frac{1}{4}N_{\Tilde{J}}(\xi, ds_0)
\end{eqnarray}

\noindent
where $\Tilde{\nabla}$ is from \eqref{connection} and  $N_{\Tilde{J}}$ is the Nijenhuis tensor, see \cite{ms2} Remark 3.1.2. 

As $\nabla_{\xi}\Tilde{J} = 0$ for all $\xi \in \Omega^0(S^2, s_0^*(TP))$, \eqref{connection} becomes $\Tilde{\nabla} = \nabla$.  A formula for $N_{\Tilde{J}}(X, Y)$ (which can be found in \cite{ms2} Lemma C.7.1) is given by
\begin{eqnarray*}
N(X, Y) = (J\nabla_YJ - \nabla_{JY}J)X - (J\nabla_XJ-\nabla_{JX}J)Y.
\end{eqnarray*}

\noindent
Thus $\nabla_{\xi}\Tilde{J} = 0$ also implies the Nijenhuis tensor vanishes, so that $D_{s_0}$ reduces to 
\begin{eqnarray}\label{differential}
D_{s_0}\xi = \frac{1}{2}\big(\nabla \xi + \Tilde{J}(s_0)\nabla \xi \circ j\big).
\end{eqnarray}

The complex bundle $s_0^*(TP)$ splits as $TS^2 \oplus \nu_{s_0}$ with $\nu_{s_0} = s_0^*(T^{Vert}P)$.  A trivialization for $\nu_{s_0}$ is given by the path $\{D\phi_t^H\}$, and we are assuming $D\phi_t^H(x_0) \equiv Id$.  Furthermore, $\Tilde{J}$ along this section is the constant product $J_0 \times j$, and so the complex bundle $(\nu_{s_0}, \Tilde{J}^{Vert})$ is trivial, and the connection $\nabla$ on this bundle is also trivial.  Thus we may split $s_0^*(TP)$ as a sum of complex line bundles $\oplus_0^n = L_i$, with $L_0$ corresponding to $TS^2$ and we have $c_1(L_0) = 2$ and $c_1(L_i) = 0$ for $i \ne 0$.

Moreover by \eqref{differential} $D_{s_0}$ preserves this splitting.   This shows that \eqref{differential} gives the formula for the standard Cauchy-Riemann operator.  The vertical portion of $D_{s_0}$ acts on a trivial bundle, and we see that the vertical portion of $D_{s_0}$ is surjective.
\end{proof}

\bigskip
\noindent \emph{Proof of Theorem \ref{totallydegnonmax}}:
\bigskip

Let $\Mm_1([s_0], \Tilde{J})$ be the space of equivalence classes $[u, z]$ of simple holomorphic sections in class $[s_0]$ with one marked point.  Here two holomorphic section maps $(u, z)$ and $(u', z')$ are called equivalent if there is $f \in PSL(2, \C)$ so that 
\begin{eqnarray*}
u' = u \circ f \mbox{ and } f(z') = z.
\end{eqnarray*}

Identify $x_0$ with its image over $0 \in D_+$.  There is only one $\Tilde{J}$ holomorphic curve in class $[s_0]$ passing through $x_0$, and all other sections through $x_0$ have larger energy.  Since all stable $\Tilde{J}$ holomorphic maps through $x_0$ must involve a section, there can be no bubbling.  

\noindent
We have the evaluation map
\begin{eqnarray*}
& ev : \Mm_1([s_0], \Tilde{J}) \times S^2 \rightarrow P \mbox{, by} & \\
& ev([u, z]) = u(z). &
\end{eqnarray*}

Given $\Mm_1(A,J)$, the moduli space of $J$ holomorphic curves $u:S^2 \rightarrow M$ representing $A \in H_2(M)$ and a submanifold $X \subset M$, one may consider the space $ev^{-1}(X)$.  This is referred to as the ``cutdown" moduli space and consists of elements of $\Mm_1(A,J)$ which send the marked point to $X$.  Referring to this space as $\Mm_1^{Cut}(A,J,X)$, in order to use such a cutdown moduli space, three conditions must be satisfied:
\begin{itemize}
\item $\Mm^{Cut}(A, J,X)$ must be compact

\item Every curve in $\Mm^{Cut}(A, J,X)$ must be regular

\item The differential of the evaluation map must be transverse to $X$.  
\end{itemize}

We consider the cutdown moduli space given by $ev^{-1}((x_0,0)) \subset \Mm_1([s_0], \Tilde{J})$ with $0 \in D_+$.  Note that as $\Mm_1([s_0], \Tilde{J})$ had been quotiented out  by $PSL(2, \C)$, $\Mm_1^{Cut}([s_0], \Tilde{J},(x_0,0))$ consists of a single map.

The tangent space to $\Mm([s_0], \Tilde{J})$ can be identified with $ker D_{u} \subset \Omega^0(S^2, u^*(T^{Vert}P))$, and the differential of the evaluation map at the point $(u, w)$ is given by
\begin{eqnarray*}
dev_{u,w} (\xi) = \xi(w).
\end{eqnarray*}

This is surjective at  $\Mm_1^{Cut}([s_0], \Tilde{J},(x_0,0))$ if, given any $v \in T_{s_0(w)}^{Vert}P$, there is $\xi \in \Omega^0(S^2, s_0^*(T^{Vert}P))$ satisfying
\begin{eqnarray*}
\xi(0) = v \mbox{, and } D_{s_0}\xi = 0.
\end{eqnarray*}

But as $s_0^*(T^{Vert}P)$ has been shown to be a trivial holomorphic bundle, we may choose $\xi$ to be a constant section.  One can see from \eqref{differential} that $D_{s_0}(\xi) = 0$ if $\xi$ is constant, and $dev_{s_0,0}$ must then be surjective.  

As $dim(\Mm_1([s_0], \Tilde{J})) = 2n + 2c_1^{Vert}([s_0]) = 2n$, the fact that $\Mm_1([s_0], \Tilde{J})$ contains $2n$ dimensions worth of constant sections allows us to apply Lemma \ref{smallest} to see that $D_{max}$ is open.  This completes the proof of Theorem \ref{totallydegnonmax}

\newpage
\noindent \emph{Proof of Corollary \ref{totallydegenerateintro}}:  
\bigskip

Suppose $x_0$ is a totally degenerate global maximum.  Then as in Theorem \ref{totallydegnonmax}, $x_0$ is an interior point of $F_{max}$.  As $F_{max}$ consists of global maxima, we must have $\partial F_{max} \subset F_{max}$.  Since $M$ is assumed to be connected, $F_{max}$ must equal $M$.  As our loop $\phi_t^H$ was assumed to be nonconstant, this contradiction completes the argument. 

\begin{remark}
One may note that elements of $\partial D_{max}$ will still be totally degenerate fixed points.  While Proposition \ref{regular} continues to hold at totally degenerate fixed points which are not maxima, elements of  $\partial D_{max}$ will not necessarily be local maxima, so that Lemma \ref{smallest} fails to hold.  The difficulty for these boundary sections is that nearby sections need not be constant.    Thus one cannot conclude that elements of $\partial D_{max}$ are interior points.
\end{remark}

\end{subsection}
\end{section}


\bibliography{bibliography} 
\end{document}